\documentclass[reqno]{amsart}
\usepackage[dvipsnames]{xcolor}
\usepackage{amssymb,amscd,bbm,verbatim,graphicx,mathtools,enumitem}
\usepackage{textcomp}
\usepackage[mathscr]{euscript}
\usepackage[colorlinks]{hyperref}
  \usepackage{setspace}
\newtheorem{thm}{Theorem}[section]

\theoremstyle{definition}
\newtheorem{defn}[thm]{Definition}

\newtheorem{isp}[thm]{The Invariant Subspace Problem}
\numberwithin{equation}{section}

\newcommand{\SP}{\operatorname{span}}
\begin{document}
\onehalfspacing
\title[]{Refuting a recent proof of the Invariant Subspace Problem}%
\author{Ahmed Ghatasheh}%

\address{Department of Mathematics, Philadelphia University, P.O.Box: 1 Amman - Jordan 19392}
\email{aghatasheh@philadelphia.edu.jo}
\email{med.ghatasheh@gmail.com}

\date{\today}%

%\thanks{Name of \TeX{} file: \texttt{LRG.tex}}%


\begin{abstract}
This article demonstrates that the recent proof of the invariant subspace problem, as presented by Khalil et al., is incorrect.
\end{abstract}
\maketitle

This article consists of two sections. The first section provides an introduction to the basics of Hilbert spaces, while the second section discusses the proof of the invariant subspace problem presented by Khalil et al. \cite{axioms13090598}. Readers who are already familiar with the fundamentals of Hilbert spaces may choose to focus on Theorems \ref{86w53gbd763} and \ref{htsio937ys} in the first section and then proceed to the second section, as these theorems are crucial for demonstrating the incorrectness of Khalil et al.'s proof \cite{axioms13090598}.

\section{Basics of Hilbert Spaces}

We assume that all vector spaces are defined over the complex numbers. We denote the set of complex numbers by $\mathbb{C}$, and the set of natural numbers by $\mathbb{N}$.

Given a nonempty subset $A$ of a vector space $V$, we use the notation $\SP\left(A\right)$ for the subspace generated by $A$. Recall that $\SP\left(A\right)$ is the intersection of all subspaces of $V$ that contain $A$. Given $\alpha_1,\alpha_2,\dots,\alpha_n$ in $\mathbb{C}$ and $u_1,u_2,\dots,u_n$ in $A$, we call $\sum_{j=1}^n\alpha_ju_j$ a finite linear combination of vectors from $A$. Recall that $\SP\left(A\right)$ consists of all possible finite linear combination of vectors from $A$.

The vector space $V$ is said to be finite-dimensional if $V=\SP\left(A\right)$ for some finite set $A\subset V$. Otherwise, we say that $V$ is infinite-dimensional.

An inner product on a vector space $V$ is a function that assigns to each pair $(u,v)\in V\times V$ a complex number $\langle u,v\rangle$ such that for all $u$, $v$, and $w$ in $V$ and all $\alpha$ and $\beta$ in $\mathbb{C}$ we have
\begin{enumerate}[leftmargin=.65cm]
\item $\langle\alpha u+\beta v,w\rangle=\alpha\langle u,w\rangle+\beta\langle v,w\rangle$,
\item $\overline{\langle u,v\rangle}=\langle v,u\rangle$, and
\item $\langle u,u\rangle\geq0$ with equality only if $u=0$.
\end{enumerate}

By saying that $V$ is an inner product space, we mean that $V$ is a vector space equipped with an inner product $\langle .,.\rangle$. For any $u\in V$, we define $\|u\|=\sqrt{\langle u,u\rangle}$. Recall that $|\langle u,v\rangle|\leq\|u\|\|v\|$ for any $u$ and $v$ in $V$. This is known as the Cauchy-Schwarz inequality. For any $u$ and $v$ in $V$ and any $\alpha\in\mathbb{C}$, the following statements hold:
\begin{enumerate}[leftmargin=.65cm]
\item $\|u\|\geq0$ with equality if and only if $u=0$.
\item $\|\alpha u\|=|\alpha|\|u\|$.
\item $\|u+v\|\leq\|u\|+\|v\|$.
\end{enumerate}
Thus, $\|.\|$ defines a norm on $V$, and this norm induces a topology on $V$, where the basic open sets are $\{y\in V:\|y-x\|<\epsilon\}$ for any $x\in V$ and $\epsilon>0$.

A finite sequence $\left(u_j\right)_{j=1}^n$ in a vector space $V$ is said to be linearly independent if the only solution to the equation $\sum_{j=1}^n\alpha_ju_j=0$ is $\alpha_1=\alpha_2=\dots=\alpha_n=0$.

If $V$ has an inner product $\langle.,.\rangle$, the sequence $\left(u_j\right)_{j=1}^n$ is said to be orthogonal if $\langle u_j,u_k\rangle=0$ for all $1\leq j<k\leq n$. Finally, it is said to be orthonormal if it is orthogonal and $\|u_j\|=1$ for all $1\leq j\leq n$. The following are some important observations that can be verified directly.
\begin{enumerate}[leftmargin=.65cm]
\item If $\left(u_j\right)_{j=1}^n$ is an orthogonal sequence of nonzero vectors in $V$, then it is linearly independent.
\item If $\left(u_j\right)_{j=1}^n$ is orthonormal, then $\|\sum_{j=1}^n\alpha_ju_j\|^2=\sum_{j=1}^n|\alpha_j|^2$ for any sequence $\left(\alpha_j\right)_{j=1}^n$ in $\mathbb{C}$. 
\end{enumerate}

\begin{defn}[Independence, Orthogonality, and Orthonormality of Sequences]
A sequence $(u_j)_{j=1}^{\infty}$ in a vector space $V$ is said to be linearly independent if $\left(u_j\right)_{j=1}^n$ is linearly independent for each $n\in\mathbb{N}$. If $V$ is an inner product space, then $(u_j)_{j=1}^{\infty}$ is said to be orthogonal if $\langle u_j,u_k\rangle=0$ for all $j\neq k$ in $\mathbb{N}$. Finally, it is said to be orthonormal if it is orthogonal and $\|u_j\|=1$ for all $j\in\mathbb{N}$.
\end{defn}

If $V$ is an infinite-dimensional vector space, then it is straightforward to verify that it contains a linearly independent sequence $(u_j)_{j=1}^{\infty}$.

\begin{thm}[Gram-Schmidt Process]\label{hdgtyu9276d}
Let $(u_j)_{j=1}^{\infty}$ be a linearly independent sequence in an infinite-dimensional inner product space $V$. Set $w_1=u_1$ and
$$w_{n+1}=u_{n+1}-\sum_{j=1}^n\frac{\langle u_{n+1},w_j\rangle}{\|w_j\|^2}  w_j,$$
for each $n\in\mathbb{N}$. Then the following statements hold:
\begin{enumerate}[leftmargin=0.65cm]
\item $(w_j)_{j=1}^{\infty}$ is an orthogonal sequence of nonzero vectors in $V$.
\item $\SP\left(\{u_1,u_2,\dots,u_n\}\right)=\SP\left(\{w_1,w_2,\dots,w_n\}\right)$ for all $n\in\mathbb{N}$.
\item $\SP\left(\{u_1,u_2,\dots\}\right)=\SP\left(\{w_1,w_2,\dots\}\right)$.
\item If we set $e_j=w_j/\|w_j\|$ for each $j\in\mathbb{N}$, then $\left(e_j\right)_{j=1}^{\infty}$ is an orthonormal sequence that satisfies $\SP\left(\{u_1,u_2,\dots,u_n\}\right)=\SP\left(\{e_1,e_2,\dots,e_n\}\right)$ for all $n\in\mathbb{N}$ and $\SP\left(\{u_1,u_2,\dots\}\right)=\SP\left(\{e_1,e_2,\dots\}\right)$.
\end{enumerate}
\end{thm}

It follows from the Gram-Schmidt process that every infinite-dimensional inner product space contains an orthonormal sequence.

\begin{thm}[Finite-Dimensional Subspaces]
If $S$ is a finite-dimensional subspace of an inner product space $V$, then it is closed.
\end{thm}
\begin{proof}
See Theorem 2.4-3 in Kreyszig \cite{2007introductory}.
\end{proof}
\begin{defn}[Convergence in Inner Product Spaces]
Let $(x_n)_{n=1}^{\infty}$ be a sequence in an inner product space $V$, and let $x\in V$.
\begin{enumerate}[leftmargin=0.65cm]
\item $(x_n)_{n=1}^{\infty}$ is said to be a Cauchy sequence if for each $\epsilon>0$, there exists $N_{\epsilon}\in\mathbb{N}$ such that $\|x_n-x_m\|<\epsilon$ for all $n,m\geq N_{\epsilon}$.
\item $(x_n)_{n=1}^{\infty}$ is said to converge to $x$ if $\lim_{n\to\infty}\|x_n-x\|=0$.
\item $\sum_{j=1}^{\infty}x_j$ is said to converge to $x$, written as $\sum_{j=1}^{\infty}x_j=x$, if the sequence $(\sum_{j=1}^{n}x_j)_{n=1}^{\infty}$ converges to $x$.
\item $(x_n)_{n=1}^{\infty}$ is said to converge weakly to $x$ if $\lim_{n\to\infty}\langle x_n,y\rangle=\langle x,y\rangle$ for every $y\in V$, or equivalently $\lim_{n\to\infty}\langle x_n-x,y\rangle=0$ for every $y\in V$
\end{enumerate}
\end{defn}

Any sequence in an inner product space $V$ can converge to at most one vector in $V$. A similar statement holds for any sequence that converges weakly. In addition,
if $(x_n)_{n=1}^{\infty}$ converges to $x$, then it converges weakly to $x$. This can be shown using the Cauchy-Schwarz inequality.

Finally, the inner product space $V$ is said to be a Hilbert space if every Cauchy sequence in $V$ converges to some vector in $V$.

\begin{thm}\label{6387ydjw}
Let $\left(e_n\right)_{n=1}^{\infty}$ be an orthonormal sequence in an infinite-dimensional inner product space $V$. Suppose that $\left(\alpha_n\right)_{n=1}^{\infty}$ is a sequence in $\mathbb{C}$ and that $u\in V$. The following statements hold:
\begin{enumerate}[leftmargin=0.65cm]
\item If $u=\sum_{j=1}^{\infty}\alpha_je_j$, then $\alpha_m=\langle u,e_m\rangle$ for all $m\in\mathbb{N}$ and $\sum_{j=1}^{\infty}|\alpha_j|^2<\infty$.
\item If $V$ is a Hilbert space and $\sum_{j=1}^{\infty}|\alpha_j|^2<\infty$, then $\sum_{j=1}^{\infty}\alpha_je_j$ is convergent.
\end{enumerate}
\end{thm}

\begin{proof}
To prove the first claim, let $m\in\mathbb{N}$ be given. Observe that for any $n\geq m$, we have
$$\alpha_m=\left\langle\sum_{j=1}^{n}\alpha_je_j,e_m\right\rangle.$$
Thus
$$\alpha_m=\lim_{n\to\infty}\left\langle\sum_{j=1}^{n}\alpha_je_j,e_m\right\rangle=\langle u,e_m\rangle,$$
where the second equality is justified by the fact that $(\sum_{j=1}^{n}\alpha_je_j)_{n=1}^{\infty}$ converges weakly to $u$.

Since $u=\sum_{j=1}^{\infty}\alpha_je_j$, there exists $M>0$ such that $\|\sum_{j=1}^n\alpha_je_j\|^2\leq M$ for all $n\in\mathbb{N}$. Hence
$$\sum_{j=1}^n|\alpha_j|^2=\Big\|\sum_{j=1}^n\alpha_je_j\Big\|^2\leq M$$
for all $n$, implying that $\sum_{j=1}^{\infty}|\alpha_j|^2<\infty$.

To prove the second claim, assume that $\sum_{j=1}^{\infty}|\alpha_j|^2<\infty$. Since $H$ is a Hilbert space, it suffices to show that $\big(\sum_{j=1}^{n}\alpha_je_j\big)_{n=1}^{\infty}$ is a Cauchy sequence. If $p<q$ are in $\mathbb{N}$, then
$$\Big\|\sum_{j=1}^{q}\alpha_je_j-\sum_{j=1}^{p}\alpha_je_j\Big\|^2=\Big\|\sum_{j=p+1}^q\alpha_je_j\Big\|^2=\sum_{j=p+1}^q|\alpha_j|^2\leq\sum_{j=p}^{\infty}|\alpha_j|^2.$$
This completes the proof since $\lim_{p\to\infty}\sum_{j=p}^{\infty}|\alpha_j|^2=0$.
\end{proof}

\begin{thm}[Bessel's Inequality and Parseval's formula]\label{yduois64582}
Let $\left(e_n\right)_{n=1}^{\infty}$ be an orthonormal sequence in an infinite-dimensional inner product space $V$. For any $u\in V$, the following statements hold:
\begin{enumerate}[leftmargin=0.65cm]
\item $\sum_{j=1}^{\infty}|\langle u,e_j\rangle|^2\leq\|u\|^2$ (Bessel's Inequality).
\item $\sum_{j=1}^{\infty}|\langle u,e_j\rangle|^2=\|u\|^2$ if and only if $\sum_{j=1}^{\infty}\langle u,e_j\rangle e_j=u$ (Parseval's Formula).
\end{enumerate}
\end{thm}

\begin{proof}
For each $n\in\mathbb{N}$, set $s_n=\sum_{j=1}^n\langle u,e_j\rangle e_j$. Then
$$\langle s_n,u\rangle=\sum_{j=1}^n|\langle u,e_j\rangle|^2=\|s_n\|^2=\langle u,s_n\rangle.$$
Hence for each $n\in\mathbb{N}$, we have
$$
\|u-s_n\|^2=\|u\|^2-\langle u,s_n\rangle-\langle s_n,u\rangle+\|s_n\|^2=\|u\|^2-\sum_{j=1}^n|\langle u,e_j\rangle|^2.
$$
The above two claims follow directly from the last equation.
\end{proof}

\begin{defn}[Complete Orthonormal Sequences]\label{7780sh64d}
By a complete orthonormal sequence in an infinite dimensional inner product space $V$ we mean an orthonormal sequence $\left(e_n\right)_{n=1}^{\infty}$ for which $u=\sum_{j=1}^{\infty}\langle u,e_j\rangle e_j$ for all $u\in V$.
\end{defn}

\begin{defn}[Separable Inner Product Spaces]
An inner product space $V$ is said to be separable if it has a countable dense subset, i.e., there exists a countable set $S\subset V$ such that for each $x\in V$ there exists a sequence $(s_n)_{n=1}^{\infty}$ in $S$ (depending on $x$) such that $\lim_{n\to\infty}\|s_n-x\|=0$.
\end{defn}

\begin{thm}[Characterization of Separable Inner Product Spaces]
Let $V$ be an infinite-dimensional inner product space. Then $V$ is separable if and only if it contains a complete orthonormal sequence.
\end{thm}

\begin{proof}
For the proof, see Proposition 1.3.11 in Bennewitz et al. \cite{MR4199125}.
\end{proof}

Suppose that $V$ is an inner product space and that $S$ is a nonempty subset of $V$. A vector $u\in V$ is said to be orthogonal to $S$ if $\langle u,s\rangle=0$ for all $s\in S$.

\begin{thm}\label{11245982jhh}
Let $V$ be an inner product space and let $S$ be a subspace of $V$. The following statements hold:
\begin{enumerate}[leftmargin=.65cm]
\item $\overline{S}$ is a subspace of $V$.
\item If a vector $u\in V$ is orthogonal to $S$, then it is orthogonal to $\overline{S}$.
\end{enumerate}
\end{thm}

\begin{proof}
The proof of the first claim is trivial. To prove the second claim, let $s\in\overline{S}$ be given. Then there exists a sequence $\left(s_n\right)_{n=1}^{\infty}$ is $S$ such that $\lim_{n\to\infty}\|s_n-s\|=0$. Then $|\langle u,s\rangle|=|\langle u,s-s_n\rangle|\leq\|u\|\|s-s_n\|$ for each $n\in\mathbb{N}$. Hence $\langle u,s\rangle=0$.
\end{proof}

\begin{thm}\label{86w53gbd763}
Let $H$ be an infinite-dimensional separable Hilbert space with a complete orthonormal sequence $(\theta_n)_{n=1}^{\infty}$. If $(y_n)_{n=1}^{\infty}$ is a sequence of unit vectors in $H$ such that $\langle y_n,\theta_j\rangle=0$ for all $j\leq n$ in $\mathbb{N}$, then $(y_n)_{n=1}^{\infty}$ converges weakly to zero.
\end{thm}

\begin{proof}
This result is a consequence of Bessel's inequality, but instead, we give a direct proof. Given $x\in H$, the goal is to show that $\lim_{n\to\infty}\langle x,y_n\rangle=0$. For each $n\in\mathbb{N}$, set $r_n=\langle x,y_n\rangle y_n+\sum_{j=1}^n\langle x,\theta_j\rangle\theta_j$. Then 
$$\langle r_n,x\rangle=|\langle x,y_n\rangle|^2+\sum_{j=1}^n|\langle x,\theta_j\rangle|^2=\|r_n\|^2,$$
where the last equality follows from the fact that $y_n,\theta_1,\theta_2,\dots,\theta_n$ are orthonormal vectors. We conclude that $\langle r_n,x\rangle=\langle x,r_n\rangle=\|r_n\|^2$.
Thus
$$
0\leq\|x-r_n\|^2=\|x\|^2-\langle x,r_n\rangle-\langle r_n,x\rangle+\|r_n\|^2=\|x\|^2-\|r_n\|^2.$$
Hence, for all $n\in\mathbb{N}$, we have
\begin{equation}\label{76388dh65d76}
\|r_n\|^2=|\langle x,y_n\rangle|^2+\sum_{j=1}^n|\langle x,\theta_j\rangle|^2\leq\|x\|^2.
\end{equation}
Note that $\sum_{j=1}^{\infty}|\langle x,\theta_j\rangle|^2=\|x\|^2$ since $(\theta_n)_{n=1}^{\infty}$ is a complete orthonormal sequence, see Definition \ref{7780sh64d} and the second statement in Theorem \ref{yduois64582}. Hence, letting $n\to\infty$ in inequality \eqref{76388dh65d76} gives $\lim_{n\to\infty}\langle x,y_n\rangle=0$, which completes the proof.
\end{proof}

\begin{thm}\label{1122092}
Let $H$ be a separable Hilbert space. If $(x_n)_{n=1}^{\infty}$ is a bounded sequence in $H$, then it has a weakly convergent subsequence. In other words, there exist $x\in H$ and a strictly increasing sequence $(j_n)_{n=1}^{\infty}$ in $\mathbb{N}$ such that $(x_{j_n})_{n=1}^{\infty}$ converges weakly to $x$.
\end{thm}

\begin{proof}
For the proof, see Theorem 2.1.11 in Bennewitz et al. \cite{MR4199125}.
\end{proof}

\begin{thm}\label{htsio937ys}
Let $\left(e_j\right)_{j=1}^{\infty}$ be an orthonormal sequence in an infinite-dimensional Hilbert space $H$, and set $S=\SP\left(\{e_1,e_2,e_3,\dots\}\right)$. Then
$$\overline{S}=\left\{\sum_{j=1}^{\infty}\alpha_je_j:\left(\alpha_j\right)_{j=1}^{\infty}\text{ is a sequence in $\mathbb{C}$ satisfiying }\sum_{j=1}^{\infty}|\alpha_j|^2<\infty\right\}.$$
\end{thm}
\begin{proof}
First, denote the set on the right-hand side by $A$. Recall that the second statement in Theorem \ref{6387ydjw} informs us that if $\left(\alpha_j\right)_{j=1}^{\infty}$ is a sequence in $\mathbb{C}$ satisfying $\sum_{j=1}^{\infty}|\alpha_j|^2<\infty$, then $\sum_{j=1}^{\infty}\alpha_je_j$ is convergent. Therefore, the set $A$ is a well-defined subset of $H$.

It is clear that $A\subset\overline{S}$. To show the reverse inclusion, let $x\in\overline{S}$ be given. Bessel's inequality informs us that $\sum_{j=1}^{\infty}|\langle x,e_j\rangle|^2\leq\|x\|^2<\infty$. Hence $\sum_{j=1}^{\infty}\langle x,e_j\rangle e_j$ is in $A$. We claim that $x=\sum_{j=1}^{\infty}\langle x,e_j\rangle e_j$, which completes the proof of the theorem.

Proving this claim requires showing that $\lim_{n\to\infty}\|x-\sum_{j=1}^n\langle x,e_j\rangle e_j\|=0$, so let $\epsilon>0$ be given. Since $x\in\overline{S}$, there exists $s\in S$ such that $2\|x-s\|<\epsilon$. Also, $s\in S$ informs us that $s=\sum_{j=1}^{p}\langle s,e_j\rangle e_j$ for some $p\in\mathbb{N}$. Note that $s=\sum_{j=1}^{n}\langle s,e_j\rangle e_j$ for all $n\geq p$ since $\langle s,e_j\rangle=0$ for all $j>p$.

Now for all $n\geq p$, we have
$$
\begin{aligned}
\Bigg\|x-\sum_{j=1}^n\langle x,e_j\rangle e_j\Bigg\|
&
\leq\left\|x-s\right\|+\Bigg\|s-\sum_{j=1}^n\langle x,e_j\rangle e_j\Bigg\|\\
&
=\left\|x-s\right\|+\Bigg\|\sum_{j=1}^n\langle s,e_j\rangle e_j-\sum_{j=1}^n\langle x,e_j\rangle e_j\Bigg\|\\
&
=\left\|x-s\right\|+\Bigg\|\sum_{j=1}^n\langle s-x,e_j\rangle e_j\Bigg\|\\
&
=\left\|x-s\right\|+\sqrt{\sum_{j=1}^n|\langle s-x,e_j\rangle|^2}\\
&
\leq\left\|x-s\right\|+\sqrt{\sum_{j=1}^{\infty}|\langle s-x,e_j\rangle|^2}.
\end{aligned}$$
Note that $\sum_{j=1}^{\infty}|\langle s-x,e_j\rangle|^2\leq\|s-x\|^2$, according to Bessel's inequality. It follows that
$$\Bigg\|x-\sum_{j=1}^n\langle x,e_j\rangle e_j\Bigg\|<\epsilon$$
for all $n\geq p$, which completes the proof.
\end{proof}

\begin{defn}[Invariant Subspaces]
Let $V$ be a vector space and $T:V\to V$ a linear operator. A subspace $S$ of $V$ is said to be invariant under $T$ if $T(S)\subset S$.

Observe that $\{0\}$, $V$, and $\text{Ker}(T)$ (the kernel of $T$) are invariant subspaces of $V$ under $T$. By the trivial subspaces of $T$ we mean $\{0\}$ and $V$.
\end{defn}

\begin{thm}\label{7640hdy3}
Suppose that $V$ is an inner product space of dimension $n\geq 2$, allowing $n=\infty$. If $T:V\to V$ is a nonzero bounded linear operator that is not one-to-one, then $\text{Ker}(T)$ is a closed nontrivial invariant subspace of $V$ under $T$.
\end{thm}

\begin{proof}
Since $\text{Ker}(T)$ is an invariant subspace that is neither $\{0\}$ nor $V$, we need only to show that $\text{Ker}(T)$ is closed. Suppose that $(x_n)_{n=1}^{\infty}$ is a sequence in $\text{Ker}(T)$ that converges to $x\in V$. The goal is to show that $x\in\text{Ker}(T)$. It follows from the boundedness of $T$ that there is $M>0$ such that $\|T(x)\|=\|T(x-x_n)\|\leq M\|x-x_n\|$ for all $n\in\mathbb{N}$. Letting $n\to\infty$ gives $\|T(x)\|=0$, which completes the proof.
\end{proof}

\begin{thm}[Invariant Subspaces of Finite-Dimensional Inner Product Spaces]\label{gsutdb11}
If $V$ is an inner product space of finite-dimension $n\geq 2$ and $T:V\to V$ is a linear operator, then there exists a nontrivial subspace of $V$ that is invariant under $T$.
\end{thm}

%\begin{proof}
%The proof is trivial if $T$ is the zero operator. In view of Theorem \ref{7640hdy3}, it suffices to assume that $T$ is one-to-one. Since $V$ is finite dimensional, $T$ may be represented by an $n\times n$ invertible matrix. Pick an eigenvalue $\lambda$ of $T$, and note that $\lambda\neq0$ since $T$ is one-to-one. Then $T(x)=\lambda x$ for some nonzero vector $x\in V$. Hence $\SP(\{x\})$ is a closed nontrivial invariant subspace of $V$ under $T$.
%\end{proof}

\begin{thm}\label{sdbwiku73}
Let $H$ be a Hilbert space and $T:H\to H$ a bounded linear operator. If $S$ is an invariant subspace of $H$, then so is $\overline{S}$.
\end{thm}

\begin{proof}
It is given that $T(S)\subset S$, and the goal is to show that $T\left(\overline{S}\right)\subset\overline{S}$. Let $y\in T(\overline{S})$ be given. Then $y=T(x)$ for some $x\in \overline{S}$. Since $x\in \overline{S}$, there exists a sequence $\left(x_n\right)_{n=1}^{\infty}$ in $S$ such that $\lim_{n\to\infty}\|x_n-x\|=0$. But $T$ is bounded, so $0=\lim_{n\to\infty}\|T(x_n)-T(x)\|=\lim_{n\to\infty}\|T(x_n)-y\|$. Observe that the sequence $\left(T(x_n)\right)_{n=1}^{\infty}$ is in $S$ since $T(S)\subset S$. Thus we conclude that $y\in\overline{S}$.
\end{proof}

\section{The Proof by Khalil Et Al.}

Here is the formulation of the invariant subspace problem.

\begin{isp}
If $H$ is an infinite-dimensional separable Hilbert space and $T:H\to H$ is a bounded linear operator, then $H$ has a closed nontrivial invariant subspace under $T$.
\end{isp}

\textcolor{blue}{We will analyze the recent proof provided by Khalil et al. \cite{axioms13090598} up to the point where we demonstrate why it is incorrect. This analysis will involve breaking the proof down into individual claims. Note that the proof we provide for each claim may differ from the proofs given by the authors.}

In view of Theorem \ref{7640hdy3}, we may assume that $\text{Ker}(T)=\{0\}$, i.e., $T$ is one-to-one.

The authors begin the proof by assuming, for the sake of contradiction, that the only closed invariant subspaces of of $H$ under $T$ are $\{0\}$ and $H$.

\begin{center}
\fbox{\begin{minipage}[t]{0.96\textwidth}
\textbf{\textcolor{blue}{First Claim:}} For each nonzero vector $h\in H$, the sequence $\left(T^{n}(h)\right)_{n=0}^{\infty}$ is linearly independent. Note that $T^0$ is the identity operator.
\end{minipage}}
\end{center}

To prove this claim observe that the vector $h$ is linearly independent since it is nonzero. If $h$ and $T(h)$ are linearly dependent, then $T(h)=\alpha h$ for some nonzero scalar $\alpha$. Thus $\SP\left(\{h\}\right)$ is a closed nonzero invariant subspace of $H$ under $T$. Hence $H=\SP\left(\{h\}\right)$, contradicting the fact the $H$ is infinite-dimensional. We conclude that $h$ and $T(h)$ are linearly independent. Now, if the sequence $\left(T^{n}(h)\right)_{n=0}^{\infty}$ is linearly dependent, then there exists $m\geq1$ such that $T^{m+1}(h)=\alpha_0 h+\alpha_1 T(h)+\dots+\alpha_mT^m(h)$ for some scalars $\alpha_0,\alpha_1,\dots,\alpha_m$. Obviously, $S=\SP\left(\{h,T(h),\dots,T^m(h)\}\right)$ is a closed nonzero invariant subspace of $H$ under $T$. This implies that $S=H$, contradicting the fact that $H$ is infinite-dimensional, and hence completes the proof of the above claim.

\begin{center}
\fbox{\begin{minipage}[t]{0.96\textwidth}
\textbf{\textcolor{blue}{Second Claim:}} For any nonzero vector $h\in H$ we have
$$\displaystyle\overline{\SP\left(\left\{h,T(h),T^2(h),\dots\right\}\right)}=H.$$
\end{minipage}}
\end{center}

For each $h\in H$ set $\mathbb{S}(h)=\SP\left(\left\{h,T(h),T^2(h),\dots\right\}\right)$. To prove the above claim, let a nonzero vector $h\in H$ be given. Then $\mathbb{S}(h)$ is a nonzero invariant subspace of $H$ under $T$. It follows from Theorem \ref{sdbwiku73} that $\overline{\mathbb{S}(h)}$ is also a nonzero invariant subspace of $H$ under $T$. But $\overline{\mathbb{S}(h)}$ is closed, so $\overline{\mathbb{S}(h)}=H$.

\begin{center}
\fbox{\begin{minipage}[t]{0.96\textwidth}
\textbf{\textcolor{blue}{Third Claim:}} Given a nonzero vector $x\in H$, there exist an orthonormal sequence $(\theta_n)_{n=1}^{\infty}$ in $H$ and a sequence $(x_n)_{n=1}^{\infty}$ in $H$ such that
\begin{enumerate}[leftmargin=.65cm]
\item $\SP\left(\{\theta_1,\theta_2,\dots,\theta_n\}\right)=\SP\left(\{T(x),T^2(x),\dots,T^n(x)\}\right)$ for each $n\in\mathbb{N}$.
\item $\SP\left(\{\theta_1,\theta_2,\theta_3,\dots\}\right)=\SP\left(\left\{T(x),T^2(x),T^3(x),\dots\right\}\right)$.
\item $x_n\neq 0$ for all $n\in\mathbb{N}$.
\item $\langle x_n, \theta_j\rangle=0$ for all $j\leq n$ in $\mathbb{N}$.
\item $x=x_n+\sum_{j=1}^n\langle x,\theta_j\rangle\theta_j$.
\end{enumerate}

Note that the authors use the notation $a_j=\langle x,\theta_j\rangle$ for each $j\in\mathbb{N}$.
\end{minipage}}
\end{center}

Since $\left(T^n(x)\right)_{n=1}^{\infty}$ is linearly independent (according to the \textbf{\textcolor{blue}{First Claim}}), Theorem \ref{hdgtyu9276d} (Gram-Schmidt process) informs us that there is an an orthonormal sequence $(\theta_n)_{n=1}^{\infty}$ such that
\begin{enumerate}[leftmargin=.65cm]
\item $\SP\left(\{\theta_1,\theta_2,\dots,\theta_n\}\right)=\SP\left(\{T(x),T^2(x),\dots,T^n(x)\}\right)$ for each $n\in\mathbb{N}$, and
\item $\SP\left(\{\theta_1,\theta_2,\theta_3,\dots\}\right)=\SP\left(\left\{T(x),T^2(x),T^3(x),\dots\right\}\right)$.
\end{enumerate}

For each $n\in\mathbb{N}$, set $x_n=x-\sum_{j=1}^n\langle x,\theta_j \rangle \theta_j$. Then it is straightforward to verify that the sequence $(x_n)_{n=1}^{\infty}$ satisfies the last three properties in the \textbf{\textcolor{blue}{Third Claim}}.

\begin{center}
\fbox{\begin{minipage}[t]{0.96\textwidth}
\textbf{\textcolor{blue}{Fourth Claim:}} For each $n\in\mathbb{N}$, set $y_n=x_n/\|x_n\|$. Then, there exists a subsequence $\left(y_{k_n}\right)_{n=1}^{\infty}$ of $\left(y_n\right)_{n=1}^{\infty}$ such that $\left(y_{k_n}\right)_{n=1}^{\infty}$ converges weakly to a nonzero vector $y\in H$.
\end{minipage}}
\end{center}

The \textcolor{blue}{Fourth Claim} is actually \textcolor{red}{\textbf{FALSE}}. The authors rely on this claim to obtain a contradiction. They begin their proof of the above claim by assuming, without loss of generality, that $\left(y_n\right)_{n=1}^{\infty}$ weakly converges to $y$. They then claim that a bounded linear functional $f:H\to\mathbb{R}$ that satisfies $f(y_n)=2\left(1-2^{-n}\right)$ for all $n\in\mathbb{N}$, can be constructed using $(\theta_n)_{n=1}^{\infty}$ and $(y_n)_{n=1}^{\infty}$, and nothing more. They proceed to deduce that $y\neq0$ since $f(y)=\lim_{n\to\infty}f(y_n)=2$. What is wrong with their argument is that $f$ is not \textcolor{blue}{well-defined}.

We will refute the \textcolor{blue}{Fourth Claim} by showing that the sequence $\left(y_n\right)_{n=1}^{\infty}$ is forced to converge weakly to zero. Of course, this informs us that every subsequence of $\left(y_n\right)_{n=1}^{\infty}$ is also forced to converge weakly to zero.  We begin by noting that since $T(x)$ is a nonzero vector in $H$, choosing $h=T(x)$ in the \textcolor{blue}{Second Claim} gives
$$H=\overline{\SP\left(\{T(x),T^2(x),T^3(x),\dots\}\right)}=\overline{\SP\left(\{\theta_1,\theta_2,\theta_3,\dots\}\right)}.$$
It follows from Theorem \ref{htsio937ys} that
$$H=\left\{\sum_{j=1}^{\infty}\alpha_j\theta_j:\left(\alpha_j\right)_{j=1}^{\infty}\text{ is a sequence in $\mathbb{C}$ satisfiying }\sum_{j=1}^{\infty}|\alpha_j|^2<\infty\right\}.$$
Thus, if $h\in H$ is given, then $h=\sum_{j=1}^{\infty}\alpha_j\theta_j$ for some sequence $\left(\alpha_j\right)_{j=1}^{\infty}$ in $\mathbb{C}$ that satisfies $\sum_{j=1}^{\infty}|\alpha_j|^2<\infty$. It follows from the first statement in Theorem \ref{6387ydjw} that $\alpha_j=\langle h,\theta_j\rangle$ for all $j\in\mathbb{N}$. This shows that $\left(\theta_j\right)_{j=1}^{\infty}$ is a complete orthonormal sequence in $H$, see Definition \ref{7780sh64d}.

It follows from the fourth statement in the \textcolor{blue}{Third Claim} that $\langle y_n,\theta_j\rangle=0$ for all $j\leq n$ in $\mathbb{N}$. Then Theorem \ref{86w53gbd763} informs us that the sequence $(y_n)_{n=1}^{\infty}$ weakly converges to zero.

\section*{Acknowledgement}
I would like to thank Steven Redolfi and Rudi Weikard for reading this article and providing valuable comments and feedback.

 \end{document}